\documentclass[12pt]{amsart}

\usepackage{aliascnt}
\usepackage{amsmath}
\usepackage{amsthm}
\usepackage{amsfonts}
\usepackage{amssymb}
\usepackage[colorlinks=true]{hyperref}
\usepackage{stackrel}
\usepackage[all]{xy}
\usepackage{stmaryrd}

\DeclareMathOperator{\Ad}{Ad}
\DeclareMathOperator{\ad}{ad}
\DeclareMathOperator{\AOp}{AOp}
\DeclareMathOperator{\Aut}{Aut}
\DeclareMathOperator{\Bun}{Bun}
\DeclareMathOperator{\Gr}{Gr}
\DeclareMathOperator{\Ind}{Ind}
\DeclareMathOperator{\Ker}{Ker}
\DeclareMathOperator{\Map}{Map}
\DeclareMathOperator{\Pic}{Pic}
\DeclareMathOperator{\Res}{Res}
\DeclareMathOperator{\rk}{rk}
\DeclareMathOperator{\Spec}{Spec}
\DeclareMathOperator{\Vir}{Vir}

\newcommand{\A}{\mathbf{A}}
\newcommand{\C}{\mathbf{C}}
\newcommand{\bO}{\mathbf{O}}
\newcommand{\bP}{\mathbf{P}}

\newcommand{\cA}{\mathcal{A}}
\newcommand{\cE}{\mathcal{E}}
\newcommand{\cH}{\mathcal{H}}
\newcommand{\cL}{\mathcal{L}}
\newcommand{\cM}{\mathcal{M}}
\newcommand{\cO}{\mathcal{O}}
\newcommand{\cV}{\mathcal{V}}

\newtheorem{thm}{Theorem}
\newaliascnt{lm}{thm}
\newaliascnt{prop}{thm}
\newaliascnt{cor}{thm}
\newtheorem{lm}[lm]{Lemma}
\newtheorem{prop}[prop]{Proposition}
\newtheorem{cor}[cor]{Corollary}
\aliascntresetthe{lm}
\aliascntresetthe{prop}
\aliascntresetthe{cor}

\theoremstyle{remark}
\newtheorem*{rem}{Remark}

\theoremstyle{definition}
\newtheorem*{defn}{Definition}

\begin{document}

\title{Virasoro constraints in Drinfeld--Sokolov hierarchies}
\author{Pavel Safronov}
\date{}
\email{\href{mailto:psafronov@math.utexas.edu}{psafronov@math.utexas.edu}}
\address{Department of Mathematics, University of Texas, Austin, TX 78712}
\begin{abstract}
We describe a geometric theory of Virasoro constraints in generalized Drinfeld--Sokolov hierarchies. Solutions of Drinfeld--Sokolov hierarchies are succinctly described by giving a principal bundle on a complex curve together with the data of a Higgs field near infinity. String solutions for these hierarchies are defined as points having a big stabilizer under a certain Lie algebra action. We characterize principal bundles coming from string solutions as those possessing connections compatible with the Higgs field near infinity. We show that tau-functions of string solutions satisfy second-order differential equations generalizing the Virasoro constraints of 2d quantum gravity.
\end{abstract}
\maketitle



\setcounter{section}{-1}
\section{Introduction}
\subsection{}

Drinfeld--Sokolov hierarchies are generalizations of the so-called Korteweg--de Vries (KdV) hierarchy -- that is, a family of nonlinear PDEs -- to an arbitrary semisimple Lie group $G$. For example, the first nontrivial equation of the $G=SL_2\C$ hierarchy is the well-known KdV equation
\[4u_t = u_{xxx} + 6uu_x\]
for a function $u(t,x)$. The whole KdV hierarchy can be conveniently rephrased in terms of the Schr\"{o}dinger operator $L=\partial_x^2+u(x)$ known as the Lax operator in this case. One can similarly obtain the Gelfand--Dikii, or the nKdV hierarchies, by considering an $n$-th order differential operator $L=\partial_x^n+u_{n-2}(x)\partial_x^{n-2}+...+u_0(x)$. Drinfeld and Sokolov's idea \cite{DS} was to replace the differential operator $L$ by a special connection known as a $G$-oper. Thus, the phase space of Drinfeld--Sokolov hierarchies is the moduli space of $G$-opers on the formal disk, and one recovers the original nKdV phase space when $G=SL_n\C$. We will follow a description of \textit{generalized} Drinfeld--Sokolov hierarchies due to Ben-Zvi and Frenkel \cite{BZF1}, which starts with a data of a smooth projective curve $X$, a semisimple group $G$ and a maximal torus $A\subset LG$ of the loop group called a Heisenberg subgroup. The phase space of these hierarchies is the moduli space of affine opers on the formal disk. Considering $X=\C\bP^1$ and $A$ the principal Heisenberg subgroup one recovers the original Drinfeld--Sokolov phase space.

Alternatively, we can also generalize the picture from a different point of view. Using the inverse scattering method one can replace the ``coordinate'' $u(x)$ on the phase space of the KdV hierarchy by the spectral properties of the corresponding Schr\"{o}dinger operator $L$. In the situation we are considering here we get a rank 2 vector bundle on $\A^1$, whose fiber at $\lambda\in\A^1$ is the eigenspace $\Ker(L-\lambda)$. This vector bundle can be extended to the whole $\C\bP^1$ in a canonical way; moreover, the extension carries an extra structure near infinity: the vector bundle comes as a pushforward from the trivial line bundle on a fully-ramified spectral cover $\Spec \C\llbracket\lambda^{-1/2}\rrbracket\rightarrow \Spec \C\llbracket\lambda^{-1}\rrbracket$. This picture has an extension to generalized Drinfeld--Sokolov hierarchies by considering the moduli space (known as the abelianized Grassmannian) of $G$-bundles on $X$ together with a reduction to the Heisenberg subgroup $A$ near a marked point. Ben-Zvi and Frenkel defined a natural isomorphism between the oper description of the Drinfeld--Sokolov hierarchy and the spectral description via bundles on curves. The spectral description of integrable hierarchies has an attractive quality that the time evolution is linearized: namely, the Drinfeld--Sokolov flows simply change the data of a reduction near the marked point.

The abelianized Grassmannian is a scheme of infinite type, so to get a handle on this space we would like to be able to describe certain special points. These points also admit an interesting geometric description we are about to describe.

While studying Gelfand--Dikii hierarchies, Krichever considered special Lax operators $L$, such that there is another differential operator $M$, which commutes with $L$ and has a coprime order. He called these Lax operators \textit{algebro-geometric}, since they can be recovered from the data of the spectral curve $\Spec \C[L, M]$ together with a line bundle over it, whose fiber over $(\lambda, \mu)$ is the joint eigenspace $\Ker(L-\lambda)\cap\Ker(M-\mu)$. The reader is referred to \cite{Mu} for a gentle introduction to the classical approach to these solutions.

Instead, we will follow a more general approach to algebro-geometric solutions of generalized Drinfeld--Sokolov hierarchies outlined in \cite{BZF1}. They are described in terms of abstract Higgs bundles of Donagi, such that the cameral cover near the marked point on $X$ is identified with that of the Heisenberg subgroup $A$. The orbits of Drinfeld--Sokolov flows are finite-dimensional and can be identified with torsors over Picard and Prym varieties associated to the cameral covers. Informally, one can say that algebro-geometric solutions admit a sufficiently big stabilizer under the Drinfeld--Sokolov flows. Essentially, the only known solutions of Drinfeld--Sokolov hierarchies are algebro-geometric.

In addition to algebro-geometric solutions we have the so-called \textit{string} solutions, the simplest of which is $u(x,t) = -\frac{2x}{3t}$ in the KdV case. The corresponding Lax operators $L$ admit a differential operator $P$, such that $[P, L] = 1$ (the string equation). The famous conjecture of Witten, proved by Kontsevich, states that the partition function of 2d quantum gravity (Gromov--Witten potential of a point) satisfies Virasoro constraints and the equations of the KdV hierarchy. In fact, this partition function is the tau-function associated to the solution written above. As Kac and Schwarz have shown \cite{KS} (see also \cite{AvM}), in the case of nKdV hierarchies, imposing the string equation, the KdV hierarchy, and certain analyticity assumptions is enough to fix the solution uniquely, so that the tau-function will satisfy the Virasoro constraints. The reference \cite{Sch} discusses more general string solutions of nKdV hierarchies.

In this paper we generalize the notion of string solutions from nKdV hierarchies to generalized Drinfeld--Sokolov hierarchies for arbitrary base curves and arbitrary Heisenbergs. Just as algebro-geometric solutions had a sufficiently big stabilizer under the Drinfeld--Sokolov flows, we define a bigger Lie algebra action on the phase space for \textit{varying} curves; the string solutions then possess a big stabilizer under the action of the latter Lie algebra. The goal of the paper is to understand what kind of geometry is behind these solutions and to prove the Virasoro constraints (certain second-order differential equations) on the tau-functions of string solutions.

The plan of the paper is as follows. In \autoref{sect:background} we review background material on Heisenberg subgroups and abstract Higgs bundles. We then proceed to define the phase space of generalized Drinfeld--Sokolov hierarchies together with Drinfeld--Sokolov flows for a fixed curve $X$ and in moduli. A geometric definition of tau-functions for generalized Drinfeld--Sokolov hierarchies was missing in the literature, so we give their brief description in \autoref{sect:tau} relating, in particular, to a more algebraic definition in the case $X=\C\bP^1$ that can be found in \cite{Wu}. Note, that in this paper we only consider a linear action of the Virasoro algebra on the tau-function. We do not know a geometric meaning of the non-linear symmetries that appear in \textit{loc. cit.}

In \autoref{sect:ag} we discuss algebro-geometric solutions for generalized Drinfeld--Sokolov hierarchies in a way that will be easily generalized to string solutions. We also discuss the linear differential equations the algebro-geometric tau-functions satisfy.

\autoref{sect:string} is the main part of the paper. We define string solutions as points having a big stabilizer under the Heisenberg-Virasoro Lie algebra acting on the phase space. We then give a geometric description of string solutions (\autoref{prop:virconnection}) as spectral bundles with a connection, which has a standard structure near the marked point. Generalizing a theorem of F. Plaza Mart\'{i}n \cite[Theorem 3.1]{PM} from the case of nKdV hierarchies to generalized Drinfeld--Sokolov hierarchies, we prove that algebro-geometric and string conditions are mutually exclusive if the base curve $X$ has genus 0 (\autoref{prop:virag}). To obtain Virasoro constraints on string tau-functions, we first define the Sugawara embedding for arbitrary Heisenbergs. Using the Sugawara currents we define an action of the negative part of the Virasoro algebra $H^0(X\backslash\infty, T_X)$ on the space of tau-functions. When $G$ is simply-laced, the space of tau-functions is one-dimensional. Moreover, all finite-dimensional representations of the negative part of the Virasoro algebra in this case are trivial since it is simple, thus we get Virasoro constraints, i.e. second-order differential equations on string tau-functions. This is the content of \autoref{prop:tauvir}. Note, that we prove the Virasoro constraints in the greatest possible generality: for all simply-laced groups, all Heisenbergs and all curves. We conclude the paper with a description of string solutions in terms of opers.

\subsection{Acknowledgements}
First and foremost, the author would like to thank his advisor, David Ben-Zvi, without whose guidance and careful explanations this work would not be completed. The author also thanks Kevin Costello for useful discussions and the referee for pointing out a mistake in the definition of algebro-geometric solutions.

\section{Drinfeld--Sokolov hierarchies}

\label{sect:background}

\subsection{Loop groups}

\subsubsection{}

Let $D=\Spec \C\llbracket z\rrbracket$ be the formal disk and $D^\times=\Spec\C(\!(z)\!)$ the punctured disk.

For $G$ a complex connected semisimple group we define the loop group $LG$ to be the group of maps $\Map(D^\times, G)$ under pointwise multiplication. It has a natural structure of an ind-scheme. Similarly, we define the positive loop group $LG_+$ to be the group scheme of loops $\Map(D, G)$ that extend over the puncture.

\begin{rem}
It will be useful to think of $LG$ and $LG_+$ as group schemes that we denote by $\underline{G}$ with constant fibers $G$ over $D^\times$ and $D$ respectively.
\end{rem}

\subsection{Heisenberg subgroups}

\label{sect:heisenberg}

\subsubsection{}
As all Cartan subgroups of $G$ are conjugate, the variety of all Cartan subgroups is isomorphic to $G/N(H)$ for some fixed Cartan subgroup $H\subset G$, where $N(H)$ is the normalizer of $H$ in $G$.

The projection $\underline{\cH}:=G\times^{N(H)}H\rightarrow G/N(H)$ realizes $\underline{\cH}$ as the universal group scheme whose fiber at a point $x\in G/N(H)$ is the corresponding Cartan subgroup. We have a map $G\times^{N(H)} H\rightarrow G$ given by $(g, h)\mapsto ghg^{-1}$ which realizes $\underline{\cH}$ as a subgroup scheme of the trivial group scheme $G\times G/N(H)$ over $G/N(H)$.

\begin{defn}
A \textit{Heisenberg subgroup} $A\subset LG$ is a maximal torus in $LG$ as a group scheme over $D^\times$.
\end{defn}

In other words, one has a classifying map $C_A\colon D^\times\rightarrow G/N(H)$ and the group scheme $\underline{A}\rightarrow D^\times$ is obtained as a pullback $C_A^*\underline{\cH}$.

For $H\subset G$ a Cartan subgroup we have the homogeneous Heisenberg subgroup $A=LH$. The corresponding classifying map $C_A\colon D^\times\rightarrow G/N(H)$ is the constant map to the base point. In particular, it extends over the puncture. In general, however, the limiting group will not be a Cartan subgroup, but, still, in nice situations it will be a centralizer of a regular element.

The variety $G/N(H)$ parametrizes centralizers of regular semisimple elements of $\mathfrak{g}$, the Lie algebra of $G$. It has a compactification $\widetilde{G/N(H)}$ which parametrizes centralizers of regular elements of $\mathfrak{g}$. Moreover, it is equipped with a universal abelian group scheme which we also denote by $\underline{\cH}$. The action of $G$ on $G/N(H)$ extends to the compactification so that the group scheme $\underline{\cH}$ is $G$-equivariant.

\begin{defn}
A Heisenberg subgroup $A$ is called \textit{regular} if the corresponding classifying map $C_A\colon D^\times\rightarrow G/N(H)$ extends to a map $C_{A_+}\colon D\rightarrow \widetilde{G/N(H)}$.
\end{defn}

In this case there is an abelian group scheme $\underline{A}_+\rightarrow D$ extending $\underline{A}\rightarrow D^\times$.

\subsubsection{Higgs bundles and spectral curves}

We refer the reader to \cite{DG} for a comprehensive treatment of abstract Higgs bundles, here we will only sketch the necessary basic facts.

Let $X$ be a curve.

\begin{defn}
A \textit{regular Higgs bundle} $(P, \sigma)$ on $X$ is a $G$-torsor $P\rightarrow X$ together with a $G$-equivariant map $\sigma\colon P\rightarrow \widetilde{G/N(H)}$ called a Higgs field.
\end{defn}

For example, suppose $P$ is the trivial $G$-torsor $G\times D$ on the disk. Then by $G$-equivariance the map $\sigma$ is uniquely determined by its restriction to \[\{e\}\times D\subset G\times D\rightarrow \widetilde{G/N(H)}\] for $e\in G$ the identity element. In other words, a Higgs field on the trivial $G$-torsor on the disk is the same as a regular Heisenberg.

There is a natural $W=N(H)/H$-torsor $G/H\rightarrow G/N(H)$ which extends to a ramified $W$-cover $\widetilde{G/H}\rightarrow \widetilde{G/N(H)}$. The pullback of this $W$-cover along $\sigma$ is $G$-equivariant and hence descends to a $W$-cover of $X$ known as the \textit{cameral cover} $X[\sigma]$. We call a Higgs field $\sigma$ \textit{unramified} if the image of $\sigma$ lies inside of $G/N(H)\subset \widetilde{G/N(H)}$. In this case the cameral cover is unramified, and so is a $W$-torsor.

Given a regular Heisenberg $A$, we denote by $D[A]$ the cameral cover of the associated Higgs field on $D$. Similarly, for a not necessary regular Heisenberg we denote by $D^\times[A]$ the cameral cover of the associated Higgs field on the punctured disk $D^\times$.

\subsubsection{} For $A$ a regular Heisenberg, $G$-torsors induced from $\underline{A}_+$-torsors carry an extra structure. Indeed, suppose $E$ is an $\underline{A}_+$-torsor on the disk $D$. Then the induced $G$-torsor $P=E\times^{\underline{A}_+} G$ carries a subbundle of $\ad P$ consisting of regular centralizers. This family gives a section of $P\times^G \widetilde{G/N(H)}\rightarrow D$, which is the same as a $G$-equivariant map $P\rightarrow \widetilde{G/N(H)}$. But this is nothing else but a Higgs field. In fact, we have the following statement:
\begin{prop}[Donagi--Gaitsgory]
The category of $\underline{A}_+$-torsors on $D$ is equivalent to the category of Higgs bundles $(P, \sigma)$ together with an isomorphism of $W$-covers $D[\sigma]\stackrel{\sim}\rightarrow D[A]$.
\end{prop}

The Higgs field corresponding to $A$ on the punctured disk is unramified, so the description of $\underline{A}$-torsors is more explicit:
\begin{prop}
The following three categories are equivalent:
\begin{enumerate}
\item The category of $\underline{A}$-torsors on $D^\times$.

\item The category of unramified Higgs fields $(P,\sigma)$ on $D^\times$ together with an isomorphism of $W$-torsors $D^\times[\sigma]\stackrel{\sim}\rightarrow D^\times[A]$.

\item The category of $N(H)$-torsors $P_{N(H)}$ together with an isomorphism of the associated bundle \[P_{N(H)}\times^{N(H)} W\stackrel{\sim}\rightarrow D^\times[A].\]
\end{enumerate}
\end{prop}

\subsubsection{}

Given a Heisenberg $A$, we can look at the monodromy of the cameral cover around a loop in $D^\times$. It gives a conjugacy class in the Weyl group $W$. As Kac and Peterson showed (\cite{KP}, see also \cite[Corollary 5.1.7]{BZF1} for the non simply-laced case), the $LG$-conjugacy classes of Heisenbergs are labeled by the conjugacy class of the monodromy. Roughly, the argument goes as follows. The set of Heisenbergs is isomorphic to the set of maps $\Map(D^\times, G/N(H))$. Therefore, the set of $LG$-conjugacy classes of Heisenbergs is isomorphic to the set of $N(H)$-torsors on $D^\times$. The extension of groups
\[1\rightarrow H\rightarrow N(H)\rightarrow W\rightarrow 1\]
shows that the isomorphism classes of $N(H)$-torsors on $D^\times$ whose induced $W$-torsors are isomorphic to a fixed one is a torsor over the group of $H$-torsors on $D^\times$. But since $H$ is a torus, the group of $H$-torsors is trivial by Hilbert's theorem 90.

\subsubsection{}

Consider the group scheme $\underline{A}\rightarrow D^\times$ for a Heisenberg $A$. One can view it as a family of groups over $D^\times$ with each one non-canonically isomorphic to a fixed Cartan $H$. To make the isomorphism canonical, consider the variety $G/H$, the universal cameral cover of $G/N(H)$, which parametrizes Cartan subgroups $H'\subset G$ together with an isomorphism $H'\cong H$ given by conjugation. The pullback of the universal group scheme $\cH=G\times^{N(H)} H\rightarrow G/N(H)$ to $G/H$ is isomorphic to $G\times^H H$. But since $H$ is abelian, $G\times^H H\cong G/H\times H$. In other words, the pullback of the universal group scheme $\cH\rightarrow G/N(H)$ to $G/H$ is isomorphic to the constant group scheme with fiber $H$.

Therefore, the pullback of $A$ to its cameral cover $D^\times[A]$ is $\Map(D^\times[A], G)$-conjugate to the homogeneous Heisenberg $\Map(D^\times[A], H)$. Moreover, we can identify $A$ as an abstract group (that is, forgetting the embedding $A\subset LG$) with the group of $W$-equivarant gauge transformations on the cameral cover $\Map_W(D^\times[A], H)$.

If $[s]$ is the conjugacy class of the monodromy in the Weyl group, then the cameral cover $D^\times[A]$ splits as a disjoint union of $h:1$ fully-ramified covers of $D^\times$, where $h$ is the order of $s$. In particular, we can construct a Heisenberg algebra $\mathfrak{a}$ as an abstract Lie algebra from $s$ in the following way. Let $t$ be a formal coordinate. Then the action of $s$ extends from $\mathfrak{h}$ to $\mathfrak{h}(\!(t)\!)$ by the following formula:
\[s.(at^m) = (s.a)e^{2\pi i m/h} t^m,\quad a\in\mathfrak{h}.\]
The Heisenberg $\mathfrak{a}$ is then obtained as the subspace $s$-fixed vectors in $\mathfrak{h}(\!(t)\!)$.

Given an $\underline{A}$-torsor $E\rightarrow D^\times$, the pullback $E[A]$ of $E$ to its own cameral cover $D^\times[A]$ is naturally an $H$-torsor.

Being a pullback to the $W$-cover, the underlying space $E[A]$ is $W$-equivariant. However, the $H$-torsor structure is only $N(H)$-shifted $W$-equivariant \cite{DG}. Concretely, this means that
\[w.(hx) = (whw^{-1})(w.x),\quad w\in W, h\in H, x\in E[A].\]

\subsubsection{} Recall that for $\pi\colon P\rightarrow X$ a $G$-torsor the adjoint bundle $\ad P$ is the vector bundle of $G$-invariant vertical vector fields on $P$, while the Atiyah bundle $\cA_P$ is the bundle of all $G$-invariant vector fields. Explicitly,
\[\ad P = (\pi_* T_{P/X})^G, \cA_P = (\pi_* T_P)^G.\]
They fit into an exact sequence
\[0\rightarrow \ad P\rightarrow \cA_P\rightarrow T_X\rightarrow 0\]
and a connection by definition is a splitting of this sequence.

The adjoint bundle $\ad E$ still makes sense for an $\underline{A}$-torsor $E\rightarrow D^\times$; however, the usual definition of the Atiyah bundle does not make sense for group schemes. We will define it in the following way. Recall that $E[A]$, the pullback of $E$ to its cameral cover $D^\times[E]\cong D^\times[A]$, is an $H$-torsor. We have the Atiyah sequence for $E[A]$
\[0\rightarrow \ad E[A]\rightarrow \cA_{E[A]}\rightarrow T_{D^\times[A]}\rightarrow 0\]
on $D^\times[A]$. Let $p\colon D^\times[A]\rightarrow D^\times$ be the projection. Pushing forward the sequence to $D^\times$ and taking $W$-invariants, we get a sequence
\[0\rightarrow \ad E\rightarrow (p_*\cA_{E[A})^W\rightarrow T_{D^\times}\rightarrow 0.\]

We define the Atiyah bundle of $E$ to be $\cA_E=(p_*\cA_{E[A]})^W$. Note that it was crucial that $\underline{A}$ becomes constant after pulling back to a finite cover.

The Atiyah sequence of $E$ induces a canonical connection $\nabla^{ad}$ on $\ad E$ in the following way. Let $v\in T_{D^\times}$ be a vector field and consider its lift $\tilde{v}\in \cA_E$. Then one can define
\[\nabla^{ad}_v a:= [\tilde{v}, a]\]
for any element $a\in\ad E$. This is well-defined since any two lifts $\tilde{v}$ differ by an element of $\ad E$, which is abelian. The Lie algebra $H^0(D^\times, T_{D^\times})$ is the Witt algebra with generators $L_n=-z^{n+1}\frac{\partial}{\partial z}$ for an integer $n$ and the commutation relations
\[[L_n, L_m] = (n-m)L_{n+m}.\]

The Lie algebra $H^0(D^\times, \ad E)\cong \mathfrak{a}$ is independent of the choice of the torsor $E$ and, in fact, so is the connection $\nabla^{ad}$. To see this it is enough to show this for the Atiyah sequence on the cameral cover $D^\times[A]$. But every $H$-torsor on $D^\times[A]$ is trivial and the action of $T_{D^\times[A]}$ on $\ad E[A]\cong \Map(D^\times[A], \mathfrak{h})$ is then simply the action of vector fields on functions. Therefore, the action of $H^0(D^\times, T_{D^\times[A]})$ on the Heisenberg $\mathfrak{a}\cong \Map_W(D^\times[A], \mathfrak{h})$ is given by lifting the vector field to a $W$-invariant vector field on the cameral cover and acting on $\Map_W(D^\times[A], \mathfrak{h})$.

Explicitly, if $\mathfrak{a}$ is the subspace of $s$-invariants in $\mathfrak{h}(\!(t)\!)$, the Witt algebra acts by
\[L_n =-\frac{t^{nh + 1}}{h}\frac{\partial}{\partial t}.\]
The generator $L_0$ induces a gradation $-\nabla^{ad}_{L_0}$ (note the minus sign) on $\mathfrak{a}$. From the explicit formula we see that this is a $\frac{1}{h}\mathbb{Z}$-grading.

Recall that an element $w\in W$ of the Weyl group is called \textit{elliptic} if it has no fixed vectors in the Cartan subalgebra $\mathfrak{h}$. For example, Coxeter elements in $W$ are elliptic.

\begin{prop}
The degree 0 subspace of a Heisenberg $\mathfrak{a}$ is trivial iff the monodromy is elliptic.
\label{prop:heisenbergcoxeter}
\end{prop}
\begin{proof}
Suppose $\nabla^{ad}_{L_0} v=0$ for some $v\in\mathfrak{a}\cong\Map_W(D^\times[A], \mathfrak{h})$. The tangent bundle to $D^\times[A]$ is trivial, so the map $v$ is annihilated by any vector field.

Therefore, $v$ is a locally-constant map $D^\times[A]\rightarrow \mathfrak{h}$. Moreover, the image is fixed by the monodromy. We see that the degree 0 subspace of the Heisenberg coincides with the space of fixed vectors by the monodromy of $A$.
\end{proof}

\subsubsection{Examples}
\begin{enumerate}
\item Consider the homogeneous Heisenberg $A=LH$. In this case $\nabla^{ad}=d$ and so the gradation given by $-\nabla^{ad}_{L_0}$ is the homogeneous one. Clearly, $\mathfrak{h}\subset L\mathfrak{h}$ is the subspace of flat sections and the monodromy is trivial.

\item Choose Chevalley generators $e_1,...,e_n,f_1,...,f_n,h_1,...,h_n$ of $\mathfrak{g}$. If $\theta\in\mathfrak{h}^\vee$ is the maximal root, then $e_0=f_\theta\otimes z$, $f_0=e_\theta\otimes z^{-1}$ and $\{e_i,f_i,h_i\}_{i=1}^n$ are the Chevalley generators of $L\mathfrak{g}$.

Let $p_{-1}=\sum_{i=0}^n f_i$ be the principal nilpotent element of $L\mathfrak{g}$. Then the principal Heisenberg $A$ by definition is the centralizer of $p_{-1}$ in $LG$.

Let $h_{Cox}$ be the Coxeter number, i.e. the order of the Coxeter element. The principal gradation on $L\mathfrak{g}$ is given by assigning degree $1/h_{Cox}$ to $e_i$, degree 0 to $h_i$ and degree $-1/h_{Cox}$ to $f_i$. For example, $p_{-1}\in\mathfrak{a}$ has degree $-1/h_{Cox}$ and multiplication by $z $ increases the degree by $1$.

One has an explicit action of the Witt algebra given by
\[-\nabla^{ad}_{L_0} = z\frac{\partial}{\partial z} + \sum_{j=1}^n c_j h_j,\]
where $c_j$ are solutions to $\sum_{j=1}^n c_j a_{ij} = 1/h_{Cox}$ for $\{a_{ij}\}$ the Cartan matrix and $h_{Cox}$ the Coxeter number.

The monodromy of the cameral cover is a Coxeter element, so the Heisenberg has no degree 0 vectors. Indeed, suppose that $p\in\mathfrak{a}$ is an element of degree 0. Then $p=\sum_{j=1}^n b_j h_j$ for some $b_j\in\C$. Since $p$ centralizes $p_{-1}$,
\[-\sum_{i=1,j=1}^n b_j a_{ij} f_i + \sum_{j=1}^n b_j \theta(h_j) e_\theta=0.\]

Since $\mathfrak{g}$ is semisimple, the Cartan matrix $a_{ij}$ is nondegenerate. So, $b_j=0$, i.e. $p=0$.
\end{enumerate}

\subsection{Grassmannians}

In this section we define the phase space of the generalized Drinfeld--Sokolov hierarchy together with a collection of vector fields which make it into an integrable system.

\subsubsection{}

The input data for a generalized Drinfeld--Sokolov hierarchy consists of:
\begin{itemize}
\item A smooth projective curve $X$.

\item A semisimple simply-connected group $G$.

\item A choice of a regular Heisenberg subgroup $A\subset LG$.
\end{itemize}

The Drinfeld--Sokolov Grassmannian $\widehat{\Gr}_g$ is the moduli stack whose $S$-points parametrize the following data:
\begin{itemize}
\item A family of genus $g$ smooth projective curves $X\rightarrow S$ together with a marked point $\infty\colon S\rightarrow X$. We denote by $X_0=X\backslash\infty$ the affine part and $X_\infty$ the spectrum of the completed local ring at infinity.

\item A local coordinate $z\colon X_\infty\stackrel{\sim}\rightarrow D\times S$.

\item A $G$-torsor $P\rightarrow X$ together with a reduction of $P|_{X_\infty}$ to an $z^*\underline{A}_+$-torsor $E\rightarrow X_\infty$.

\item A trivialization $E\stackrel{\sim}\rightarrow z^*\underline{A}_+$.
\end{itemize}

To not obscure the notation, the local coordinate $z$ will be implicit from now on.

Note, that a trivialization of $E$ induces a trivialization of $P|_{X_\infty}$, hence the Grassmannian $\widehat{\Gr}_g$ is independent of the choice of a Heisenberg.

Let $\Gr_g$ be the moduli space obtained from $\widehat{\Gr}_g$ by forgetting the trivialization of $E$. We have the maps
\[\widehat{\Gr}_g\rightarrow \Gr_g\rightarrow \widehat{\cM}_{g, 1},\]
where $\widehat{\cM}_{g,1}$ is the moduli space of genus $g$ curves with a marked point together with a choice of a local coordinate. We denote by $\widehat{\Gr}_X$ and $\Gr_X$ the fibers of $\widehat{\Gr}_g\rightarrow \widehat{\cM}_{g,1}$ and $\Gr_g\rightarrow \widehat{\cM}_{g,1}$ respectively over a curve $X\in\widehat{\cM}_{g,1}$. The space $\Gr_X$ is known as the \textit{abelianized Grassmannian}.

One can explicitly realize $\widehat{\Gr}_X$ and $\Gr_X$ as the spaces of cosets
\[\widehat{\Gr}_X = LG_-\backslash LG,\quad \Gr_X = LG_-\backslash LG/A_+,\]
where $LG_- = \Map(X_0, G)$ is the ind-group of loops that extend away from the marked point.

\begin{rem}
If two Heisenbergs $A$ and $A'$ are $LG_+$-conjugate, the corresponding Drinfeld--Sokolov phase spaces are canonically isomorphic and the isomorphism intertwines the flows. However, as noted in \cite{BZF1}, there are continuous families of $LG_+$-conjugacy classes of Heisenbergs in the same $LG$-conjugacy class.
\end{rem}

\subsubsection{Examples}

All known examples of Drinfeld--Sokolov hierarchies start with a genus 0 curve, so let $X=\C\bP^1$. In this case we have an open dense subset $\Gr^{big\ cell}_X\subset \Gr_X$ consisting of trivializable $G$-torsors $P$.

We have the following standard choices of Heisenbergs:
\begin{enumerate}
\item $A=LH$ is the homogeneous Heisenberg. For $G=SL_2\C$ one gets the non-linear Schr\"{o}dinger hierarchy.

\item $A$ is the principal Heisenberg. Then the big cell $\Gr^{big\ cell}_X$ parametrizes $G$-opers on the disk and we recover the original description of Drinfeld and Sokolov. For example, the case $G=SL_2\C$ corresponds to the KdV hierarchy and $G=SL_n\C$ to its generalizations known as Gelfand--Dikii or nKdV hierarchies.
\end{enumerate}

The reader is referred to \cite{Kr} for explicit coordinates on Drinfeld--Sokolov phase spaces.

\subsection{Line bundles}

\subsubsection{}
Let $\langle,\rangle$ be an $\mathrm{Ad}$-invariant bilinear form on $\mathfrak{g}$ taking even values on the coroots. It defines a central extension
\[1\rightarrow \C^\times\rightarrow\widehat{LG}\rightarrow LG\rightarrow 1.\]
On the level of Lie algebras, one has $\widehat{L\mathfrak{g}}\cong L\mathfrak{g}\oplus\C$ as vector spaces together with the Lie bracket
\[[(a, \alpha), (b,\beta)] = [a, b] + \Res_{z=0}\langle a, db\rangle.\]

In particular, we see that the central extension is trivial when restricted to $LG_-$, i.e. we have an embedding $LG_-\subset \widehat{LG}$. Therefore, we get a line bundle
\[LG_-\backslash\widehat{LG}\rightarrow LG_-\backslash LG.\]

Since the central extension restricted to $LG_+$ is also split, thus defined line bundle is $LG_+$-equivariant, i.e. it descends to a line bundle $\cL$ on the moduli space of bundles $\Bun_G(X)\cong LG_-\backslash LG/LG_+$. Similarly, it descends to the abelianized Grassmannian $\Gr_X$.

\subsubsection{}

So far we have defined line bundles on $\Gr_X$, and the question arises whether they extend to the whole Grassmannian $\Gr_g$. We have a section $\widehat{\cM}_{g,1}\rightarrow \Gr_g$ given by considering the trivial torsor, which gives a splitting of the pullback map $\Pic \widehat{\cM}_{g,1}\rightarrow \Pic \Gr_g$.

Pick a curve $X\in\widehat{\cM}_{g,1}$. Then we have an exact sequence
\[1\rightarrow \Pic \widehat{\cM}_{g,1}\rightarrow \Pic\Gr_g\rightarrow \Pic\Gr_X,\]
where the last map is the restriction to the fiber over $X$. Laszlo \cite{La} has shown that one can carry out a construction of the line bundle $\cL$ in families thus obtaining a splitting $\Pic\Gr_X\rightarrow \Pic\Gr_g$, which produces a split exact sequence
\[1\rightarrow \Pic \widehat{\cM}_{g,1}\rightarrow \Pic\Gr_g\rightarrow \Pic\Gr_X\rightarrow 1.\]

\begin{rem}
In types A and C the splitting $\Pic\Gr_X\rightarrow \Pic\Gr_g$ can be obtained as the determinant line bundle of a fundamental representation.
\end{rem}

\subsection{Lie algebra actions}

\subsubsection{}
Recall that the Atiyah sequence
\[0\rightarrow \ad E\rightarrow \cA_E\rightarrow T_{D^\times}\rightarrow 0\]
of an $\underline{A}$-torsor $E\rightarrow D^\times$ is the sequence of $W$-invariants of the Atiyah sequence of the corresponding $H$-torsor on the cameral cover $D^\times[A]$. As $H$-torsors on $D^\times[A]$ are trivializable, the Atiyah sequence is non-canonically split, so we have an identification
\[\cA:= H^0(D^\times, \cA_E)\cong H^0(D^\times, T_{D^\times}) \ltimes \mathfrak{a}.\]

There is an action of the Lie algebra $\cA$ on the Grassmannian $\widehat{\Gr}_g$ given by deforming both the curve $X$ and the torsor $E$. Note, that the action of $\cA$ does not descend to the Grassmannian $\Gr_g$ since $\cA$ and $\mathfrak{a}_+$ do not commute.

The subalgebra \[\cA^0:=H^0(D^\times, \ad E)\cong\mathfrak{a}\] of vertical vector fields acts along the fibers of $\widehat{\Gr}_g\rightarrow\widehat{\cM}_{g, 1}$, i.e. it preserves $\widehat{\Gr}_X$. Moreover, the action descends to the abelianized Grassmannian $\Gr_X$.

This action of $\cA^0$ generates the Drinfeld--Sokolov flows, while $\cA$ represents the so-called Orlov--Shulman extended symmetries.

\subsubsection{}
As line bundles $\cL$ on $\widehat{\Gr}_g$ correspond to central extensions of $LG$, the action of $\cA^0\cong\mathfrak{a}$ lifts canonically to the line bundle by restricting the central extension to the Heisenberg. Similarly, we have a central extension $\widehat{\cA}$, which lifts the action of $\cA$ to the line bundle.

\subsection{Tau-function}

\label{sect:tau}

\subsubsection{}
The Lie algebra $\cA$ exponentiates to an ind-group $K=\Aut(D^\times)\ltimes A$. Similarly, there is a central extension $\widehat{K}$ which lifts the action of $K$ to the line bundle $\cL$ on $\widehat{\Gr}_g$.

Consider the action and projection maps
\[\widehat{\Gr}_g\times \widehat{K}\stackrel[p]{a}\rightrightarrows \widehat{\Gr}_g.\]

Then the action of $\widehat{K}$ on $\cL$ can be expressed as the data of an isomorphism $p^*\cL\cong a^*\cL$.

Fix a nonzero section $\sigma\in H^0(\Gr_g, \cL)\cong H^0(\widehat{\Gr}_g, \cL)^{A_+}$. Then one defines a rational function on $\widehat{\Gr}_g\times\widehat{K}$ called the \textit{extended tau-function} $\overline{\tau}$ as
\[\overline{\tau}=\frac{p^*\sigma}{a^*\sigma}.\]

Explicitly, one has
\[\overline{\tau}(P, g) = \frac{\sigma(g^{-1} P)}{g^{-1}\sigma(P)}.\]

Thus, it measures the failure of the section $\sigma$ to be $\widehat{K}$-invariant. Since $\sigma$ is $A_+$-invariant, the extended tau-function is a function on $\widehat{\Gr}_g\times\widehat{K}/A_+$.

Fixing a point $P\in\widehat{\Gr}_g$ we define the \textit{tau-function} of $P$ as the restriction of the extended tau-function $\overline{\tau}$ to the slice $\{P\}\times\widehat{A}/A_+$. It is a rational function on $\widehat{A}/A_+$.

\subsubsection{}
Although $A$ is abelian, its central extension $\widehat{A}$ is not. Its failure is measured by the commutator pairing $c\colon A\times A\rightarrow \C^\times$ given by arbitrarily lifting the elements of $A$ to the central extension and computing the commutator.

Suppose $a\in A_+$. Then
\[\tau_{aP}(g) = \frac{\sigma(g^{-1}aP)}{g^{-1}\sigma(aP)} = \frac{a\sigma(g^{-1}P)}{g^{-1}a\sigma(P)} = c(a, g^{-1})\frac{a\sigma(g^{-1}P)}{ag^{-1}\sigma(P)} = c(a, g^{-1})\tau_P(g).\]

In particular, the tau-function is not well-defined as a function on $\Gr_g$ and is instead a section of an associated line bundle to $\widehat{\Gr}_g\rightarrow \Gr_g$.

\subsubsection{} There is an equivalent, more algebraic, way to write down the tau-function. Fix a point $P\in \widehat{\Gr}_X$ and pull back the line bundle $\cL$ to the orbit $\widehat{A}$ of the Drinfeld--Sokolov flows. $\sigma$ defines a section of $\cL$; in particular, it gives a connection $d\log\sigma$ with regular singularities at the zeros of $\sigma$.

We define a meromorphic 1-form $d\log\tau$ on $\widehat{A}$ by
\[d\log\tau[a] = d\log\sigma[a] - v_a,\]
where $a\in\widehat{\mathfrak{a}}$ and $v_a\in \cA_{\cL}$ is the vector field exhibiting the action of $\widehat{A}$ on $\cL$.

The connection $d\log\sigma$ splits the sequence
\[0\rightarrow \cO\rightarrow \cA_\cL\rightarrow T_{\widehat{A}}\rightarrow 0\]
whenever $\sigma$ does not vanish, which happens generically when $X$ has genus 0. Then $d\log\tau[a]$ is simply the image of $-v_a$ under the splitting $\cA_\cL\rightarrow \cO$. This is the definition of tau-functions which can be found in \cite[Section 3.2]{Wu}.

\section{Algebro-geometric solutions}

\label{sect:ag}

\subsection{Geometric description}

\subsubsection{}

Recall that the Grassmannian $\widehat{\Gr}_X$ is isomorphic to $LG_-\backslash LG$, where the corresponding element of $LG$ is the transition function from the formal disk near infinity to the affine part. Using the trivialization of the torsor $E$ on the formal disk, we can identify $H^0(D^\times, \Ad P)\cong LG$ and $H^0(D^\times, \Ad E)\cong A$.

\begin{lm}
The stabilizer of the $LG$-action on $\widehat{\Gr}_X$ is $H^0(X_0, \Ad P)\subset H^0(D^\times, \Ad P)$.
\end{lm}
\begin{proof}
Suppose an element $g\in LG$ fixes a coset $[\gamma]\in LG_-\backslash LG$. Explicitly, this means that there is an element $\gamma\in LG_-$, such that $\gamma g = \gamma_-\gamma$ or, equivalently, $g = \gamma^{-1}\gamma_-\gamma$. But $\gamma$ is the transition function from the formal disk to the affine part, so we get that $g\in H^0(X_0, \Ad P)$.
\end{proof}

\begin{cor}
The stabilizer of the $\mathfrak{a}$-action on $\widehat{\Gr}_X$ is $\cA^0_{stab}=\mathfrak{a}\cap H^0(X_0, \ad P)$.
\end{cor}

\subsubsection{}

Both $\mathfrak{a}$ and $H^0(X_0, \ad P)$ are torsion-free $\cO(X_0)$-modules. Therefore, the stabilizer $\cA^0_{stab}$ is also a torsion-free module; moreover, it is finitely-generated since $H^0(X_0, \ad_P)$ is so. Since $X_0$ is smooth of dimension 1, the stabilizer $\cA^0_{stab}$ is a finitely-generated projective module, so we can localize it to get a vector bundle $c$ on $X_0$:
\[c = \cA^0_{stab}\otimes_{\cO(X_0)} \cO_{X_0}.\]

By construction, $c|_{D^\times}\subset \ad E$ and $c\subset \ad P|_{X_0}$ are subsheaves. In particular, the first inclusion implies that the rank of $c$ is bounded from above by the rank of $\ad E$ which coincides with $\rk G$, the rank of the group $G$. Note that the inclusion $c\subset \ad P|_{X_0}$ is not in general a subbundle.

\begin{prop}
Suppose $X$ has genus 0. Then the inclusion $c\subset \ad P|_{X_0}$ is a subbundle.
\end{prop}
\begin{proof}
We have to show that for every point $x\in X_0$ the map
\[\cA^0_{stab}\otimes_{\cO(X_0)} \cO(X_0) / \cO_x(X_0)\rightarrow H^0(X_0, \ad P)\otimes_{\cO(X_0)} \cO(X_0) / \cO_x(X_0)\]
is injective, where $\cO_x(X_0)$ is the ideal of functions vanishing at the point $x$.

This is equivalent to the following isomorphism:
\[\left(\cO_x(X_0) H^0(X_0, \ad P)\right)\cap \cA^0_{stab}\cong \cO_x(X_0)\cA^0_{stab}.\]

Since $X$ has genus 0, the ideal $\cO_x(X_0)$ is principal, i.e. $\cO_x(X_0)\cong f\cO(X_0)$ for some regular function $f\in \cO(X_0)$ vanishing to the first order at $x$. Suppose $s$ is an element of $\left(\cO_x(X_0) H^0(X_0, \ad P)\right)\cap \cA^0_{stab}$. Then $s\in H^0(D^\times, \ad E)$ and $s=f\tilde{s}$ for some $\tilde{s}\in H^0(X_0, \ad P)$. Therefore, $\tilde{s}\in H^0(D^\times, \ad E)$ since $f$ becomes invertible in $\cO(D^\times)$ and $H^0(D^\times, \ad E)$ is an $\cO(D^\times)$-module.
\end{proof}

\begin{defn}
A point $P\in\widehat{\Gr}_X$ is \textit{algebro-geometric} if the rank of $c$ coincides with $\rk G$ and $c\subset \ad P|_{X_0}$ is a subbundle.
\end{defn}

Clearly, this condition is independent of the trivialization of $E$, so it makes sense to consider algebro-geometric points of $\Gr_X$.

\subsubsection{}

In the algebro-geometric case we have two subbundles $c\subset \ad P|_{X_0}$ and $\ad E\subset \ad P|_D$, which are identified on the intersection. Therefore, we can glue them together to get a bundle $c$ on the whole $X$. $c$ has rank $\rk G$ and is a subbundle of abelian Lie algebras, hence it defines an extension of a Higgs field $\ad E\subset \ad P|_D$ to a regular Higgs field on $X$. The converse is also true:
\begin{prop}
A point $P\in\Gr_X$ is algebro-geometric iff the subbundle $\ad E\subset \ad P|_D$ over $D$ extends to a regular Higgs field on $P\rightarrow X$.
\end{prop}
\begin{proof}
Suppose that $c'$ is the extension of $\ad E$ to $X$. Then \[H^0(X_0, c')\subset \mathfrak{a}\cap H^0(X_0, \ad P)\cong \cA^0_{stab}.\] Therefore, $c'|_{X_0}\subset c$. Thus, the rank of $c$ is bounded from below by $\rk G$ and hence coincides with $\rk G$, i.e. the point is algebro-geometric.
\end{proof}

We see that algebro-geometric points in $\Gr_X$ can be reconstructed from the following data:
\begin{itemize}
\item A $G$-torsor $P\rightarrow X$.

\item A regular Higgs field $c\subset \ad P$.

\item An isomorphism of cameral covers $D[c|_D]\cong D[A]$.
\end{itemize}

\subsubsection{}

Consider two points $P_1,P_2\in\Gr_X$ which differ by an action of $A$. The corresponding Higgs fields $c_{P_1}$ and $c_{P_2}$ are isomorphic away from infinity, so we have an isomorphism of cameral covers $X_0[c_{P_1}]\cong X_0[c_{P_2}]$. Since both $X[c_{P_1}]$ and $X[c_{P_2}]$ are smooth projective curves, this gives an isomorphism $X[c_{P_1}]\cong X[c_{P_2}]$. Using the results of \cite{DG} we see that an $A$-orbit of an algebro-geometric point in $\Gr_X$ is a torsor over the corresponding Prym variety. In particular, it is finite-dimensional.

\subsection{Tau-function}

\subsubsection{}

Consider the diagram

\[\xymatrix{
& & 0 & 0 & \\
0 \ar[r] & \cA^0_{stab} \ar@{-->}[rd] \ar[r] & \cA^0\ar[u]\ar[r] & T_{\Gr_X}\ar[u] &\\
&& \widehat{\cA}^0\ar[u]\ar[r] & \cA_\cL\ar[u] & \\
&& \cO_{\Gr_X}\ar[u]\ar@2{-}[r] & \cO_{\Gr_X}\ar[u] \\
&& 0\ar[u] & 0\ar[u] &
}\]
where the middle column is a pullback of the rightmost column, the Atiyah sequence of $\cL$, along $\cA^0\rightarrow T_{\Gr_X}$. By the universal property there is a unique lift $\cA^0_{stab}\rightarrow\widehat{\cA}^0$, such that the composite $\cA^0_{stab}\rightarrow\widehat{\cA}^0\rightarrow \cA_\cL$ is zero.

Let $v\in \cA^0_{stab}(P)$ be any vector stabilizing a point $P\in\Gr_X$. Since $\cA^0_{stab}\rightarrow \cA_\cL$ is the zero map, $v$ preserves the fiber of $\cL$ at $P$. Therefore, the tau-function
\[\tau_P(g)=\frac{\sigma(g^{-1}P)}{g^{-1}\sigma(P)}\]
obeys a first-order differential equation
\[v\tau_P=0\]
for any $v\in \cA^0_{stab}(P)$. This simply means that the tau-function descends to a well-defined function on the orbit $\widehat{A}/A_{stab}(P)$, where $A_{stab}(P)$ is the stabilizer of the point $P$ in $\widehat{\Gr}_X$ under the action of the Heisenberg $A$. It is linear along the fibers of $\widehat{A}/A_{stab}(P)\rightarrow A/A_{stab}(P)$ and so defines a section of the line bundle dual to $\widehat{A}/A_{stab}(P)\rightarrow A/A_{stab}(P)$.

\section{String solutions}

\label{sect:string}

\subsection{Geometric description}

\subsubsection{}

Algebro-geometric points were characterized by the property that enough elements of the Heisenberg $\mathfrak{a}$ stabilize a point. Similarly, one can consider the action of the Witt-Heisenberg algebra $\cA$, which can be thought of as the algebra of first-order differential operators with coefficients in $\mathfrak{a}$.

In the algebro-geometric case we looked at \[\cA^0_{stab}=\cA^0\cap H^0(X_0, \ad P),\] so now consider \[\cA_{stab}=\cA\cap H^0(X_0, \cA_P).\] We have the following diagram:
\[
\xymatrix{
0\ar[r] & H^0(X_0, \ad P)\ar[r] & H^0(X_0, \cA_P)\ar[r] & H^0(X_0, T_X)\ar[r] & 0\\
0\ar[r] & \cA^0_{stab}\ar[r] \ar@{^{(}->}[u] & \cA_{stab}\ar[r]\ar@{^{(}->}[u] & H^0(X_0, T_X) \ar@{=}[u] &
}
\]

Note, that there is no reason to expect in general that the map $\cA_{stab}\rightarrow H^0(X_0, T_X)$ is surjective.

\begin{defn}
A point $P\in\widehat{\Gr}_X$ is \textit{string} if the map $\cA_{stab}\rightarrow H^0(X_0, T_X)$ is surjective.
\end{defn}

If $X$ has genus 0, $A$ is the principal Heisenberg and $G=SL_n\C$, Schwarz \cite{Sch} showed that this definition is equivalent to the so-called string equation, hence the name. We will return to the issue of the string equation in \autoref{sect:stringdiff}.

Note, that in the string case the sequence of Lie algebras \[0\rightarrow \cA^0_{stab}\rightarrow \cA_{stab}\rightarrow H^0(X_0, T_X)\rightarrow 0\] is non-canonically split since $T_X$ is locally-free.

\subsubsection{}

Let us pick a splitting $H^0(X_0, T_X)\rightarrow \cA_{stab}$ and consider the composite
\[H^0(X_0, T_X)\rightarrow \cA_{stab}\hookrightarrow H^0(X_0, \cA_P).\]

By definition, this means that $P$ carries a connection on the affine part $X_0$. Moreover, the connection on $D^\times$ preserves the Heisenberg in the sense that for any vectors $v\in T_{D^\times}$ and $a\in\mathfrak{a}$ we have $[\nabla_v, a]\in\mathfrak{a}$. We get the following statement:
\begin{thm}
A point $P\in\widehat{\Gr}_X$ is string iff the torsor $P$ has a connection $\nabla$ on the affine part $X_0$, such that $\nabla_{z\partial/\partial z}$ induces the canonical gradation on the Heisenberg $\mathfrak{a}$.
\label{prop:virconnection}
\end{thm}
\begin{proof}
Suppose that we have a connection $\nabla$, such that $\nabla_{z\partial/\partial z}$ induces the canonical gradation. Since the module of vector fields is free over $\cO(D^\times)$, we see that the whole action of the Witt algebra on the Heisenberg $\mathfrak{a}$ coincides with the canonical action. We want to prove that it implies that the image of $\nabla$ lands in $H^0(D^\times, \cA_E)\subset H^0(D^\times,  \cA_P)$.

Pick a splitting $s\colon H^0(D^\times, T_X)\rightarrow H^0(D^\times, \cA_E)$. Since $\nabla$ gives the canonical action of the Witt algebra on the Heisenberg, the difference $\nabla_v - s_v$ is an element $g(v)$ of $H^0(D^\times, \ad P)$ for any $v\in H^0(D^\times, T_X)$. Moreover, as both $\nabla$ and $s$ preserve the Heisenberg, $g(v)$ commutes with $\mathfrak{a}$. But then $g(v)\in\mathfrak{a}$ and hence $\nabla$ lands in $H^0(D^\times, \cA_E)$.

We see that there is a map $H^0(X_0, T_X)\rightarrow H^0(D^\times, \cA_E) \cap H^0(X_0, \cA_P)$ and hence the point is string.
\end{proof}

We get the following geometric structure on string solutions:
\begin{itemize}
\item A $G$-torsor $P\rightarrow X$ together with an $\underline{A}_+$-reduction $E\rightarrow D$.

\item A subbundle $c\subset \ad P|_{X_0}$ of abelian Lie algebras, such that $c|_{D^\times}\subset \ad E|_{D^\times}$.

\item A connection $\nabla$ on the affine part $X_0$ which preserves $c$ and gives the canonical gradation on the Heisenberg $H^0(D^\times, \ad E)$.
\end{itemize}

As in the algebro-geometric case, $c$ is the localization of $\cA^0_{stab}$ to a vector bundle on $X_0$. However, since $c|_{D^\times}\subset \ad E|_{D^\times}$ is not an isomorphism, we do not have a canonical way of extending $c$ to the whole curve $X$.

\subsubsection{}

If the genus of $X$ is zero, we can try to exploit the action of global vector fields $\mathfrak{sl}_2(\C)\cong H^0(X, T_X)$ on $c$.

\begin{thm}
Suppose $X$ has genus 0 and the Heisenberg $A$ has elliptic monodromy. Then string points have $c=0$.
\label{prop:virag}
\end{thm}
\begin{proof}
Let $v\in H^0(X, T_X)$ be a regular vector field with a zero of order 2 at $\infty$. In local coordinates $v=z^2\frac{\partial}{\partial z}$. The corresponding derivation of $H^0(D^\times, \cO_X)$ has order 1. Therefore, $\nabla_v$ also has order 1 on the Heisenberg $\mathfrak{a}$.

Take a nonzero element $e\in\cA^0_{stab}$. $\nabla_v$ raises its order, so for some $n$ we have $(\nabla_v)^ne\in H^0(D, \ad E)$. On the other hand, \[H^0(D, \ad E)\cap H^0(X_0, \ad P) \subset H^0(X, \ad P)\] is finite-dimensional, so, possibly increasing $n$, we can assume that $(\nabla_v)^ne=0$. Let $n$ be the minimal such exponent, so that $s=(\nabla_v)^{n-1}e\neq 0$. Then $s$ is annihilated by $\nabla$, i.e. it is a flat section. But it cannot happen by assumption on the Heisenberg and \autoref{prop:heisenbergcoxeter}. Therefore, $\cA^0_{stab}=0$.
\end{proof}

In particular, since $c=0$, we see that string points cannot be algebro-geometric.

A similar statement was obtained previously by F. Plaza Mart\'{i}n \cite[Theorem 3.1]{PM} when $G=SL_n\C$ and $A$ is the principal Heisenberg.

\subsection{Virasoro constraints}

Any vector $a\in \cA_{stab}$ annihilates the extended tau-function
\[a\overline{\tau}=0\]
since it preserves the section $\sigma$. However, unless $a\in \cA^0_{stab}$, such a constraint does not make sense on the ordinary tau-function since it has components in the direction of the Witt algebra $H^0(D^\times, T_{D^\times})$. To get rid of these components, we would like to know how the tau-function changes under the action of the Witt algebra.

The Witt algebra $H^0(D^\times, T_{D^\times})$ has a central extension $\Vir$ called the Virasoro algebra. It is a Lie algebra with generators $C$ and $L_n$ for an integer $n$ with the commutation relations
\[[L_n, L_m] = (n-m)L_{n+m} + \delta_{n, -m}\frac{C}{12}(n^3-n), \quad [L_n, C] = 0.\]

The Heisenberg $\mathfrak{a}\cong H^0(D^\times, \ad E)$ has a central extension $\widehat{\mathfrak{a}}$ given by restricting the central extension of the loop algebra to the Heisenberg.

Combining these two central extensions, we get a central extension of $H^0(D^\times, \cA_E)$ with a two-dimensional center which is isomorphic to the semi-direct product $\Vir\ltimes\widehat{\mathfrak{a}}$.

\subsubsection{Sugawara construction for twisted Heisenbergs}

In this section we define an embedding of the Virasoro algebra into the universal enveloping algebra of a Heisenberg called the Sugawara construction, e.g. see \cite{BZF2}.

Recall from \autoref{sect:heisenberg} that a Heisenberg $\mathfrak{a}$ with monodromy $[s]$ as an abstract Lie algebra can be obtained as the subspace of $s$-invariants in $\mathfrak{h}(\!(t)\!)$, where $s\in W$ acts on $\mathfrak{h}(\!(t)\!)$ by
\[s.(at^m) = (s.a)e^{2\pi im/h}t^m,\quad a\in\mathfrak{h}.\]

Similarly, one can obtain the central extension $\widehat{\mathfrak{a}}$ as the subspace of $s$-invariants in $\widehat{\mathfrak{h}(\!(t)\!)}$, where $s$ acts trivially on the central element.

Consider a decomposition
\[\mathfrak{h}=\bigoplus_{m=0}^{h-1}\mathfrak{h}_m\]
into eigenspaces of $s$, where $s$ acts on $\mathfrak{h}_m$ by $\exp(2\pi im/h)$. We denote the dimension $d_\alpha=\dim\mathfrak{h}_{\alpha\bmod h}$.

Since $s$ is orthogonal with respect to the inner product $\langle,\rangle$ on $\mathfrak{h}$, the spaces $\mathfrak{h}_m$ and $\mathfrak{h}_{h-m}$ are naturally paired. Let $\{a^i_m\}_i$ be a basis of $t^m\mathfrak{h}_m$ and let $\{a^{\bar{i}}_{-m}\}_i$ be the dual basis of $t^{-m}\mathfrak{h}_{h-m}$.

Let us define the universal enveloping algebra $U_k(\widehat{\mathfrak{a}})$ as the quotient $U(\widehat{\mathfrak{a}})/(\mathbf{K}-k\cdot 1)$, where $\mathbf{K}$ is the central element of $\widehat{\mathfrak{a}}$. We can similarly define the universal enveloping algebras $U_c(\Vir)$ and $U_{k,c}(\Vir\ltimes\widehat{\mathfrak{a}})$.

The Sugawara currents we are about to define will involve infinite expressions in the elements $a^i_\alpha$, so they belong to a completion of the universal enveloping algebra $U_k(\widehat{\mathfrak{a}})$ in the ``positive'' direction:
\[\widehat{U}_k(\widehat{\mathfrak{a}}) = \varprojlim_n U_k(\widehat{\mathfrak{a}}) / (\widehat{\mathfrak{a}}_n U_k(\widehat{\mathfrak{a}})),\]
where we denote $\widehat{\mathfrak{a}}_n = t^{hn}\widehat{\mathfrak{a}}_+$. Similarly, we can define a completion $\widehat{U}_{k,c}(\Vir\ltimes\widehat{\mathfrak{a}})$.

\begin{prop}
The elements of $\widehat{U}_k(\widehat{\mathfrak{a}})$
\[L^S_n=\frac{1}{kh}\sum_{\alpha<nh/2, i}a^i_{hn-\alpha} a^{\bar{i}}_\alpha + \frac{1}{2kh}\sum_i a^i_{hn/2} a^{\bar{i}}_{hn/2} - \delta_{n, 0}\frac{1}{4h^2}\sum_{0<l<h} d_ll(h-l)\]
obey the Virasoro commutation relations
\[[L^S_n, L^S_m] = (n - m)L^S_{n+m} - \delta_{n, -m}\frac{\dim\mathfrak{h}}{12}(n^3-n).\]
\label{prop:sugawara}
\end{prop}
\begin{proof}
Let us assume that $h$, $n$ and $m$ are all odd for simplicity. Moreover, assume $m\geq0$. Other cases are treated similarly. We denote $C^{ij} =\langle a^i,a^j\rangle$.

Then we get
\begin{align*}
(kh)^2[L^S_n, L^S_m] &= \sum_{\alpha < hn/2, \beta<hm/2,i,j} [a^i_{hn-\alpha} a^{\bar{i}}_\alpha, a^j_{hm-\beta} a^{\bar{j}}_\beta]\\
&=k\sum_{\alpha,\beta,i,j} a^i_{hn-\alpha} a^{\bar{j}}_\beta\alpha\delta_{\alpha,\beta-hm}\delta_{ij} + a^{\bar{i}}_\alpha a^{\bar{j}}_\beta(hn-\alpha)\delta_{hn-\alpha,\beta-hm}C^{ij}\\
&+ a^j_{hm-\beta} a^{\bar{i}}_\alpha(hn-\alpha)\delta_{hn-\alpha,-\beta}\delta_{ij} + a^j_{hm-\beta} a^i_{hn-\alpha} \alpha\delta_{\alpha,-\beta}C^{\bar{i}\bar{j}}\\
&=k\sum_{\alpha<hn/2,i} a^i_{hn-\alpha} a^{\bar{i}}_{hm+\alpha}\alpha\delta_{\alpha<-hm/2} + a^{\bar{i}}_\alpha a^i_{hn+hm-\alpha}(hn-\alpha)\delta_{\alpha>hn+hm/2}\\
&+ a^i_{hm+hn-\alpha} a^{\bar{i}}_\alpha(hn-\alpha)\delta_{\alpha<hn+hm/2} + a^{\bar{i}}_{hm+\alpha} a^i_{hn-\alpha} \alpha\delta_{\alpha > -hm/2}
\end{align*}

We see that one may combine the terms to get
\[k\sum_{\alpha<hn/2,i} a^i_{hn-\alpha} a^{\bar{i}}_{hm+\alpha}\alpha + a^i_{hm+hn-\alpha} a^{\bar{i}}_\alpha(hn-\alpha).\]

Let's shift the index of summation in the first term:
\[k\sum_{\alpha<hn/2+hm,i} a^i_{hn+hm-\alpha} a^{\bar{i}}_{\alpha}(\alpha-hm) + k\sum_{\alpha<hn/2,i}a^i_{hm+hn-\alpha} a^{\bar{i}}_\alpha(hn-\alpha).\]

We want to make both sum to go up to $\alpha=h(n+m)/2$, so split off the necessary part in the first sum:
\begin{align*}
&k\sum_{\alpha<h(n+m)/2,i} a^i_{hn+hm-\alpha} a^{\bar{i}}_{\alpha}(\alpha-hm) + k\sum_{h(n+m)/2\leq\alpha<hn/2+hm,i} a^i_{hn+hm-\alpha} a^{\bar{i}}_{\alpha}(\alpha-hm)\\
&+ k\sum_{\alpha<hn/2,i}a^i_{hm+hn-\alpha} a^{\bar{i}}_\alpha(hn-\alpha).
\end{align*}

The extra piece of the sum combines with the last term to finally give
\begin{align*}
&k\sum_{\alpha<h(n+m)/2,i} a^i_{hn+hm-\alpha} a^{\bar{i}}_\alpha(hn-hm) + k\sum_i a^i_{h(n+m)/2}a^{\bar{i}}_{h(n+m)/2}h(n-m)/2\\
&- k^2\delta_{n,-m}\sum_{0<\alpha<hm/2,i}\alpha(\alpha-hm).
\end{align*}

The first two terms combine into $(kh)^2 (n-m)L^S_{n+m}$ (for $n\neq -m$), while the last sum is
\[S:=\sum_{0<\alpha<hm/2} \alpha(\alpha - hm) d_\alpha.\]

Since $d_\alpha$ depends on $\alpha$ only modulo $h$, let us substitute $\alpha = hb + l$. Then we get
\begin{align*}
S &= \sum_{-h/2<l<h/2}d_l\sum_{-l/h< b<m/2-l/h}(hb+l)(hb-hm+l)\\
&= \sum_{-h/2<l<h/2}d_l\sum_{-l/h < b < m/2-l/h}h^2b^2 + l(l-hm) + 2hbl - h^2mb\\
&= \sum_{0<l<h/2}d_ll(l-hm) + \sum_{-h/2<l<h/2}d_l\sum_{0<b<m/2-l/h}h^2b^2 + l(l-hm) + 2hbl - h^2mb\\
&= \sum_{0<l<h/2}d_ll(l-hm) + \sum_{-h/2<l<h/2}d_l \left(h^2\frac{m^2-1}{24}m + l(l-hm)\frac{m-1}{2} + \frac{m^2-1}{8}h(2l-hm)\right)
\end{align*}

Since $d_l = d_{-l}$, only even powers of $l$ survive in the last sum. Moreover, $\sum_l d_l = \dim\mathfrak{h}$. Therefore, we get
\[S=m\sum_{0<l<h/2}d_ll(l-h) - h^2\dim\mathfrak{h}\frac{m^2-1}{12}m.\]

The second term gives the usual Virasoro cocycle, while the first term can be absorbed into a redefinition of $L^S_0$.
\end{proof}

\begin{rem}
Note the unusual ordering in the Sugawara currents. We use this particular ordering as we will be interested in representations like $\cO(A/A_+)$, which are not highest-weight, but instead dual to the highest-weight ones (namely, the vacuum representation $U\mathfrak{a}/U\mathfrak{a}_+$). This is also reflected in the minus sign in front of the central charge.
\end{rem}

\begin{rem}
The Sugawara operators do not change if we multiply the bilinear form $\langle,\rangle$ on $\mathfrak{h}$ by a number.
\end{rem}

\subsubsection{}

We get an embedding of the Virasoro algebra into $\widehat{U}_k(\widehat{\mathfrak{a}})$, i.e. we have a map
\[U_{-\dim\mathfrak{h}}(\Vir)\rightarrow \widehat{U}_k(\widehat{\mathfrak{a}}).\]

Now we want to see that the adjoint action of the Sugawara currents on the Heisenberg coincides with the canonical Witt action.

\begin{prop}
The Sugawara currents $L^S_n$ of \autoref{prop:sugawara} have the following commutation relations with elements of $\mathfrak{a}$:
\[[L^S_n, a^i_\alpha] = -\frac{\alpha}{h}a^i_{\alpha + hn}.\]
\end{prop}
\begin{proof}
For simplicity, we again assume that both $h$ and $n$ are odd. Then the commutator is
\begin{align*}
[L^S_n, a^i_\alpha] &= \frac{1}{kh}\sum_{\beta<nh/2, j}[a^j_{hn-\beta}a^{\bar{j}}_\beta, a^i_\alpha] \\
&= \frac{1}{h}\sum_{\beta < nh/2, j}(hn-\beta)\delta_{hn-\beta,-\alpha}C^{ji} a^{\bar{j}}_\beta + a^j_{hn-\beta}\beta\delta_{\beta,-\alpha}\delta_{ij}\\
&= -\frac{\alpha}{h} a^i_{\alpha+hn}.
\end{align*}
\end{proof}

\subsubsection{}

Let $L_n\in \Vir$ be the standard generators of the Virasoro algebra. The last two propositions give an embedding
\[S\colon U_{c'}(\Vir)\rightarrow \widehat{U}_{k,c}(\Vir\ltimes\widehat{\mathfrak{a}})\]
given by
\[S(L_n)=L_n - L_n^S.\]

Let us just remind the reader that the first $L_n\in\Vir$ refers to the generator of the Virasoro algebra while the second $L_n^S\in\widehat{U}_k(\widehat{\mathfrak{a}})$ refers to the Sugawara currents.

Indeed,
\begin{align*}
[S(L_n), S(L_m)] &= [L_n, L_m] + [L_n^S, L_m^S] - [L_n, L_m^S] - [L_n^S, L_m]\\
&= [L_n, L_m] - [L_n^S, L_m^S] \\
&= (n-m) S(L_{n+m}) + \delta_{n,-m}\frac{c + \dim\mathfrak{h}}{12}(n^3-n).
\end{align*}
Therefore, the central charge is
\[c' = c + \dim\mathfrak{h}.\]

Furthermore, the image of $S$ commutes with $\widehat{U}_k(\widehat{\mathfrak{a}})\subset \widehat{U}_{k,c}(\Vir\ltimes\widehat{\mathfrak{a}})$ as both $L_n$ and $L_n^S$ act in the same way on $\widehat{\mathfrak{a}}$.

\subsubsection{}

In this section $\cL$ will denote the basic ample line bundle on $\Bun_G(X)$. Let $\rho\in\mathfrak{h}^*$ be the half-sum of positive roots and $h^\vee$ the dual Coxeter number.

Recall a Borel-Weil theorem for loop groups \cite[Theorem 4]{Te}:
\begin{thm}[Teleman]
There is an isomorphism of $\widehat{L\mathfrak{g}}$-representations
\[H^0(\widehat{\Gr}_X, \cL) \cong \bigoplus_\lambda L(\lambda) \otimes H^0(\Bun_G(X), \cL\otimes \cV_\lambda),\]
where $L(\lambda)$ is a level 1 irreducible highest-weight representation with highest weight $\lambda\in\mathfrak{h}^*$, $\cV_\lambda$ is the evaluation bundle corresponding to the $\mathfrak{g}$-representation with highest weight $\lambda$ and the summation goes over the weights $\lambda$, such that $\lambda+\rho$ is inside the positive alcove at level $1 + h^\vee$.
\end{thm}

From now on we assume that $G$ is simply-laced. Let $\alpha_i\in\mathfrak{h}^*$ be the simple roots of $\mathfrak{g}$. Normalize the invariant pairing $\langle,\rangle$ on $\mathfrak{h}^*$ by the condition that $\langle\alpha_i, \alpha_i\rangle = 2$ for every simple root $\alpha_i$.

Then $\lambda + \rho$ is in the positive alcove if
\begin{align*}
\langle\theta, \lambda + \rho\rangle < 1+h^\vee,\\
\langle\alpha_i, \lambda + \rho\rangle > 0.
\end{align*}

Since $\langle\theta, \rho\rangle = h^\vee - 1$ and $\langle\alpha_i,\rho\rangle=1$, we see that the inequalities can be rewritten in the form
\begin{align*}
\langle\theta, \lambda\rangle\leq 1,\\
\langle\alpha_i, \lambda\rangle \geq 0.
\end{align*}
In other words, $\lambda$ is an integral dominant weight at level 1. We see that $H^0(\widehat{\Gr}_X, \cL)$ contains only level 1 highest-weight representations with integral dominant highest weights.

Kac and Peterson \cite[p. 288]{KP} prove that the space of invariants $L(\lambda)^{A_+}$ is finite-dimensional for every integral dominant weight $\lambda$ at level 1. Finally, since $X$ is projective, the spaces of conformal blocks $H^0(\Bun_G(X), \cL\otimes \cV_\lambda)$ are finite-dimensional.

The space of tau-functions can be written as
\[H^0(\Gr_X, \cL) \cong H^0(\widehat{\Gr}_X, \cL)^{A_+} \cong\bigoplus_\lambda L(\lambda)^{A_+} \otimes H^0(\Bun_G(X), \cL\otimes \cV_\lambda).\]

There are finitely-many summands on the right-hand side and each of them is finite-dimensional. Therefore, we conclude:
\begin{prop}
If $G$ is simply-laced, the space of tau-functions $H^0(\Gr_X, \cL)$ is finite-dimensional.
\end{prop}

\subsubsection{Example}

We have
\begin{align*}
H^0(\Bun_G(X), \cL\otimes \cV_\lambda) &\cong (H^0(LG/LG_+, \cL)\otimes V_\lambda)^{L\mathfrak{g}_-}\\
&= (L(0)^*\otimes V_\lambda)^{L\mathfrak{g}_-},
\end{align*}
where $L(0)=H^0(LG/LG_+, \cL)^*$ is the irreducible level 1 representation with the zero highest weight \cite[Proposition 2.11]{Ku}. We can write $L(0)$ as a quotient of the vacuum Verma module $V(0)$
\[L(0) = V(0) / I.\]

Decompose \[\widehat{L\mathfrak{g}}\cong \mathfrak{g}\llbracket z\rrbracket \oplus z^{-1}\mathfrak{g}[z^{-1}] \oplus \C.\] Then the vacuum module is
\[V(0) = \Ind^{\widehat{L\mathfrak{g}}}_{L\mathfrak{g}_+\oplus \C} (\C v_0)\cong U(z^{-1}\mathfrak{g}[z^{-1}]) v_0.\]

Let $X=\C\bP^1$. Consider the evaluation module $V_\lambda$ at the origin $z^{-1}=0$ and pick an element $s\in H^0(\Bun_G(X), \cL\otimes \cV_\lambda)$. We can split $L\mathfrak{g}_-\cong \mathfrak{g}\oplus z^{-1}\mathfrak{g}[z^{-1}]$. The Lie algebra $z^{-1}\mathfrak{g}[z^{-1}]$ annihilates $V_\lambda$, so from the invariance of $s$ under $z^{-1}\mathfrak{g}[z^{-1}]$ we see that $s\in L(0)^*\otimes V_\lambda$ is uniquely determined by its value $s(v_0)\in V_\lambda$. The invariance of $s$ under $\mathfrak{g}$ is equivalent to the statement that $s(v_0)$ is $\mathfrak{g}$-invariant. But $V_\lambda$ is irreducible, hence $s\neq 0$ only if $V_\lambda=\C$, i.e. $\lambda=0$.

We see that $H^0(\Bun_G(\C\bP^1), \cL\otimes \cV_\lambda)=0$ unless $\lambda=0$. Therefore, the space of tau-functions is
\[H^0(\Gr_X, \cL) \cong L(0)^{A_+}\otimes H^0(\Bun_G(\C\bP^1), \cL)\cong L(0)^{A_+}.\]

Kac and Peterson computed the dimension of $L(0)^{A_+}$ in terms of the so-called \textit{defect} of the monodromy of the Heisenberg $A$. For example, for the principal or homogeneous Heisenberg $\dim L(0)^{A_+} = 1$. For $G = SL_n$ this is true for any Heisenberg.

\begin{thm}
Let $G$ be simply-laced and $X=\C\bP^1$. Suppose one of the following holds:
\begin{enumerate}
\item The group $G=SL_n$.
\item The Heisenberg $A$ is $LG$-conjugate to the homogeneous or the principal Heisenbergs.
\end{enumerate}

Then the space of tau-functions $H^0(\Gr_X, \cL)$ is one-dimensional.
\end{thm}

\subsubsection{}

The Sugawara construction gives a map
\[S\colon U(H^0(D^\times, T_{D^\times}))\rightarrow \hat{U}_{k, c}(\Vir\ltimes\widehat{\mathfrak{a}}),\]
where $c=-\dim\mathfrak{h}$.

Any vector $v\in H^0(X_0, T_X)$ acts by endomorphisms on $H^0(\widehat{\Gr}_X, \cL)$. Moreover, the Virasoro central extension splits when restricted to $H^0(X_0, T_X)$. Therefore, we have an action of $S(v)$ on the sections $H^0(\widehat{\Gr}_X, \cL)$. Since the operators $S(v)$ commute with $A_+$, the action preserves the subspace of $A_+$-invariants $H^0(\Gr_X, \cL)\subset H^0(\widehat{\Gr}_X, \cL)$. In other words, we have an action of $H^0(X_0, T_X)$ on the space of tau-functions $H^0(\Gr_X, \cL)$.

We have the following lemma \cite[Lemma 2.5.1]{BFM}:
\begin{lm}[Beilinson--Feigin--Mazur]
The Lie algebra $H^0(X_0, T_X)$ of vector fields on a curve is simple. In particular, since it is infinite-dimensional, it has no nontrivial finite-dimensional representations.
\end{lm}

Therefore, the action of $S(v)$ on tau-functions is trivial and we get the following proposition:
\begin{prop}
Any vector field $v\in H^0(X_0, T_X)$ preserves the section $\sigma$.
\end{prop}

This implies that the tau-function
\[\tau_P(g) = \frac{\sigma(g^{-1}P)}{g^{-1}\sigma(P)}\]
is invariant under the Sugawara currents acting on $g$ on the right. But since $g\in A$ and the Sugawara currents commute with $A$, it is also invariant under the Sugawara currents acting on the left.

Suppose $a\in A_+$. Then $\tau_P(ag) = c(a, g)\tau_P(g)$. Therefore, $A_+$ acts on $\tau_P(\cdot g)\in\cO(A/A_+)$ via multiplication by a function in $\cO(A/A_+)$. In contrast, elements of $A/A_+$ act by translations. Infinitesimally, it means that $\mathfrak{a}/\mathfrak{a}_+$ acts on $\tau_P(\cdot g)$ by vector fields, i.e. first-order differential operators.

Denote the projection $p\colon \cA_{stab}\rightarrow H^0(X_0, T_X)$. Combining the invariance of the tau-function under the Sugawara currents and vectors $a\in \cA_{stab}$ we get
\begin{thm}
For any vector $a\in \cA_{stab}$ we have a second-order differential equation on the tau-function:
\[(a - S(p(a)))\tau = 0.\]
\label{prop:tauvir}
\end{thm}

These are the famous Virasoro constrains of two-dimensional quantum gravity \cite{DVV}. For example, pick a global coordinate $z^{-1}$ on $\C\bP^1$ vanishing to the first order at $\infty$. Let $\{t^i_\alpha\}_{\alpha>0,i}\in\cO(A/A_+)$ be the time coordinates, so that $a^i_\alpha$ acts by $\partial/\partial t^i_{-\alpha}$ for $\alpha\leq0$, and it acts by $(-k\alpha)t^{\bar{i}}_\alpha$ for $\alpha > 0$.

Suppose $p(g)=z^2\frac{\partial}{\partial z}$ for some $g\in \cA_{stab}$. Suppose that the monodromy of the Heisenberg has order $h=2$. Then the operator
\[S(z^2\partial/\partial z) - g = L_1 - L_1^S - g =  \frac{1}{2}\sum_{\alpha\geq 0, i}(\alpha + 2)t^i_{2+\alpha}\frac{\partial}{\partial t^i_\alpha} + \frac{1}{4}\sum_i (t^i_1)^2 - \sum_{i,\alpha} g_{i,\alpha} a^i_\alpha\]
is known as the string operator, where $g_{i,\alpha}$ are some coefficients involved in the definition of $g$. Similarly,
\[S(z\partial/\partial z) - z^{-1}g = L_0 - L_0^S - z^{-1}g = \frac{1}{2}\sum_{\alpha > 0, i}\alpha t^i_\alpha\frac{\partial}{\partial t^i_\alpha} + \frac{1}{4}\sum_i\left(\frac{\partial}{\partial t^i_0}\right)^2 + \frac{d_1}{16} - \sum_{i,\alpha} g_{i,\alpha} a^i_{\alpha-2},\]
which is related to the dilaton operator.

\subsection{Differential side}
\label{sect:stringdiff}

\subsubsection{} 

The Grassmannian description of the Drinfeld--Sokolov hierarchy is a geometric reinterpretation of the original Drinfeld--Sokolov phase space, which parametrizes certain objects called affine opers. In this section we briefly recall what these are and explain what it means for an affine oper to be string. The reader is referred to \cite{BZF1} for a complete description.

Consider two groups $G$ and $K\subset G$ and let $\bO\subset\mathfrak{g}/\mathfrak{k}$ be a $K$-orbit for the adjoint action. Suppose $\cE\rightarrow C$ is a $G$-torsor and $\cE_K\rightarrow C$ a reduction to a subgroup $K\subset G$. We denote by $\cA_P$ and $\cA_{P_K}$ the corresponding Atiyah bundles.

\begin{defn}
We say that $\cE_K$ \textit{has relative position} $\bO$ with respect to a connection $\nabla\colon T_C\rightarrow \cA_\cE$ if the image of $\nabla$ is contained inside of $\cA_{\cE_K}$ up to elements of $\bO\times^K \cE_K$.
\end{defn}

If $\bO = 0$ it simply means that $\nabla$ is induced from a connection on $P_K$ in which case we say that $\cE_K$ is \textit{flat} with respect to $\nabla$.

\subsubsection{}
The \textit{Iwahori subgroup} $LG^+\subset LG_+$ is defined as the subgroup of loops which take values in a fixed Borel subgroup $B\subset G$ at the closed point.

Let $p_{-1} = \sum_{i=0}^n f_i\in L\mathfrak{g}_+$ be the principal nilpotent element and $\bO^{aff}$ the $\C^\times$-span of the $LG^+$-orbit of $p_{-1}$.

$D_t=\Spec \C\llbracket t\rrbracket$ will denote a formal disk; note that it is unrelated to the formal disk with a local coordinate $z$ in the definition of $LG$.

\begin{defn}
An \textit{affine oper} on $D_t$ is an $LG$-torsor $\cE\rightarrow D_t$ together with a connection $\nabla$, a flat reduction $\cE_{LG_-}$ to $LG_-$ and a reduction $\cE_{LG^+}$ to $LG^+$ in relative position $\bO^{aff}$ with respect to $\nabla$.
\end{defn}

Alternatively, an affine oper is a collection of the following data:
\begin{itemize}
\item A $G$-torsor $P\rightarrow X\times D_t$.

\item A $B$-reduction $P_B$ along $\infty\times D_t$.

\item A relative connection $\nabla_{oper}$ on $X_0\times D_t$ in the $D_t$ direction, with an asymptotic condition that $\nabla$ extends to a connection at $\infty$ preserving the $B$-reduction up to elements in $\bO^{aff}$.
\end{itemize}

Although the definition may seem complicated at first, it can be made quite explicit in certain cases. For example, if $X$ has genus 0, $G=GL_n\C$ and the underlying $G$-torsor on $X$ is trivial, this is simply an $n$-th order differential operator on $D_t$ with symbol 1.

The moduli space of affine opers on $D_t$ will be denoted by $\AOp_X(D_t)$. An important observation of Drinfeld and Sokolov is that affine opers have a unique $A_+$-reduction, where $A$ is the principal Heisenberg, which has relative position $\C^\times\cdot p_{-1}\in\mathfrak{a}/\mathfrak{a}_+$ with respect to the oper connection $\nabla_{oper}$.

Let $A_{-1}\subset A/A_+$ be the subgroup exponentiating the Lie algebra element $p_{-1}\in\mathfrak{a}$. Then for any point $P\in\Gr_X$ we have a family of $LG$-torsors on $A_{-1}$ together with reductions to $LG_-$ and $A_+$. Moreover, the action gives a connection $\nabla$.

\begin{prop}[Ben-Zvi--Frenkel]
Thus defined map $\Gr_X\rightarrow \AOp_X(A_{-1})$ is an isomorphism.
\end{prop}

Suppose now that $P\in\Gr_X$ is a string point. Then there is a connection $\nabla_{string}$ on $P\rightarrow X_0$, which preserves the reduction to $A$ near infinity. Hence, the oper connection $\nabla_{oper}$ is flat with respect to $\nabla_{string}$, i.e. together they form an absolute flat connection on $X_0\times A_{-1}$. Combining this with \autoref{prop:virconnection} we get a converse statement.

\begin{thm}
An affine oper $(P, P_B, \nabla)$ is string iff the relative oper connection $\nabla_{oper}$ on $X_0\times A_{-1}$ extends to an absolute flat connection preserving the $A$-reduction near infinity.
\end{thm}

\end{document}